\documentclass[11pt,reqno]{amsart}

\usepackage{amsbsy,amsfonts,amsmath,amssymb,amscd,amsthm,mathrsfs,verbatim,calc}
\usepackage{graphicx,epsfig,color}
\usepackage{enumitem}
\usepackage[a4paper,text={6.2in,8.5in},centering,includefoot,foot=1in]{geometry}

\small\normalsize
\theoremstyle{plain}
\newtheorem{mainthm}{Theorem}

\theoremstyle{definition}

\numberwithin{equation}{section}
\let\oldmarginpar\marginpar
\renewcommand\marginpar[1]{\-\oldmarginpar[\raggedleft\footnotesize \textcolor{red}{#1}]{\raggedright\footnotesize\textcolor{red}{#1}}}

\DeclareMathOperator{\IFS}{IFS}

\begin{document}
\title[Equicontinuity of the Hutchinson operator $F$ and sensitivity of  $F_-$]
{Equicontinuity of the Hutchinson operator $F$\\ and sensitivity of  $F_-$ }
\author[A. Sarizadeh]{Aliasghar Sarizadeh}
\address{{Department of Mathematics,  Ilam University}}
\address{{Ilam, Iran.}}
\email{ali.sarizadeh@gmail.com}
\email{a.sarizadeh@ilam.ac.ir}

\begin{abstract}
For an  iterated function system  $ \mathcal{F} = \{ f_1, \dots, f_k \} $
of homeomorphisms on a compact metric space $(X, d)$,
 write $ \mathcal{F}_-= \{ f_1^{-1}, \dots, f_k^{-1} \} $.
The objective of this paper is to illustrate an iterated function system $\mathcal{F}$ of homeomorphisms on 
the circle that the Hutchinson operator of 
$\mathcal{F}$ is equicontinuous, 
but the Hutchinson operator of $\mathcal{F}_-$ is sensitive.
\end{abstract}

\maketitle
\setcounter{tocdepth}{1}

\section{Introduction and preliminaries}
In various definitions of chaos, the sensitivity  is one of the important and undeniable pillars.
In 1971, \cite{Rue}, Ruelle and Takens  introduced the first precise definition for sensitivity.
To predigest, sensitivity characterizes the unpredictability of chaotic phenomena and so working 
 on sensitivity has attracted a lot of attentions from many researchers (e.g. \cite{ABC,AG,BBCDS,GW,GU,HZ}). 
 Recently, the study of  sensitivity and closely related concepts as equicontinuity has been popular  for iterated function systems and semigroup actions
 \cite{GRS,KM,Sari26}.
 An iterated function 
system, simply $\IFS$, is a finite collection of continuous self-maps of a metric space $(X,d)$. 
$\IFS$s provide a method for both generating and 
characterizing fractal images whenever the maps are contractions. IFS was 
firstly introduced and then popularized by Hutchinson \cite{Hu} and Barnsley \cite{Barn}.

We begin by recalling some definitions and notation, then formulate our
main results. Throughout this paper, $(X,d)$ stands for a compact metric space.
Let \( \mathcal{F} = \{ f_1, \dots, f_k \} \) be an $\IFS$ of homeomorphisms on \( (X, d) \). 
Write \( \mathcal{F}_-= \{ f_1^{-1}, \dots, f_k^{-1} \} \).
For $n\in \mathbb{N}$ and 
\( \omega = (\omega_1, \omega_2, \dots) \in \{ 1, \dots, k \}^{\mathbb{N}} \), define \( f^0_\omega = \text{id}_X \), 
\(f^n_\omega(x) = f_{\omega_n} \circ \dots \circ f_{\omega_1}(x)\) and \(f^{-n}_\omega(x) = f_{\omega_n}^{-1} \circ \dots \circ f_{\omega_1}^{-1}(x).
\)

 The Hutchinson operator 
\( F: \mathcal{K}(X) \to \mathcal{K}(X) \)  and its extension \( \widetilde{F}: 2^X \to 2^X \) of $\mathcal{F}$
are respectively defined by
 \( F(A) = \bigcup_{i=1}^k f_i(A) \) 
and  \( \widetilde{F}(A) = \bigcup_{i=1}^k f_i(A) \).
Recall that  $\mathcal{K}(X)$ is the hyperspace of non-empty compact subsets of $X$.
In a similar manner, the Hutchinson operator 
\( F_-: \mathcal{K}(X) \to \mathcal{K}(X) \) and  its extension \( \widetilde{F}_-: 2^X \to 2^X \) of $\mathcal{F}_-$ are defined by
 \( F_-(A) = \bigcup_{i=1}^k f_i^{-1}(A) \)
and its extension \( \widetilde{F}_-: 2^X \to 2^X \) by \( \widetilde{F}_-(A) = \bigcup_{i=1}^k f_i^{-1}(A) \).

For an $\IFS$ \(\mathcal{F}\)  on a compact metric space \((X, d)\) and \(x_1, x_2 \in X\) equip \(X\) with the metric
\begin{equation}\label{defmetric}
d_F(x_1, x_2) := \sup_{n \in \mathbb{N}} d_H(F^n(x_1), F^n(x_2)).
\end{equation}
A point \(x \in X\) is a sensitive point of \(F\) if the identity map \(\text{id}_X: (X, d) \to (X, d_F)\) is discontinuous at \(x\). 
The Hutchinson operator  \(F\) is sensitive if there exists \(\epsilon > 0\) 
such that every non-empty open subset has \(d_F\)-diameter at least \(\epsilon\). 
A point $x$  in $X$ is equicontinuous if \(\text{id}_X: (X, d) \to (X, d_F)\) is 
continuous at \(x\). We denote the set of all equicontinuous points of \(F\) by \(\mathrm{Eq}(F)\). 
The Hutchinson operator \(F\) is equicontinuous if \(\mathrm{Eq}(F) = X\).

There is an interesting question related to equicontinuity and sensitivity  
of the Hutchinson operators of $\mathcal{F}$ and $\mathcal{F}_-$.
\begin{quote}
$\mathbf{Question.}$    If  $ \mathcal{F} = \{ f_1, \dots, f_k \} $  is an $\IFS$ where    $f_i$ is a homeomorphism for each $i\in \{1,\dots,k\}$ 
and the Hutchinson operator of $\mathcal{F}$ is  equicontinuous, 
can the Hutchinson operator of $\mathcal{F}_-$ be sensitive? 
\end{quote}
The following theorem gives positive answer to the mentioned question.
\begin{mainthm}\label{T1}
There exists an IFS $\mathcal{F}$   of homeomorphisms on the circle that the Hutchinson operator of 
$\mathcal{F}$ is equicontinuous, 
for which the Hutchinson operator of $\mathcal{F}_-$ is sensitive.
\end{mainthm}
The next  result shows that equicontinuity of the Hutchinson operator $\mathcal{F}$ does not necessarily imply the sensitivity of the Hutchinson operator 
of $\mathcal{F}_-$.
\begin{mainthm}\label{T2}
There exists a non-symmetric  IFS $\mathcal{F}$  of homeomorphisms on the circle that the Hutchinson operators of 
$\mathcal{F}$ and $\mathcal{F}_-$ are equicontinuous.
\end{mainthm}
\section{ Proof of the main results}
In \cite{GRS} (Example 4.2), we present an $\IFS \mathcal{F}$ which the Hutchinson operator of $\mathcal{F}_-$ is sensitive, but 
we were incapable of proving the equicontinuity of $\mathcal{F}$.
Here, by using results in \cite{Sari24}, in conjunction with some extra assumption, we can prove the equicontinuity 
of $\mathcal{F}$.

Let  \(\mathcal{F}\)  be an $\IFS$ on a compact metric space \((X, d)\). The phase space $X$ is an attractor for $\mathcal{F}$
if for every non-empty compact set $K$, 
$$d_H(F^i(K),X)\to 0 \  \text{as}\ i\to \infty.$$ 
\begin{proof}[Proof of Theorem~\ref{T1} ]
Let $\mathcal{F}$ be  a symmetric set of homeomorphisms of the circle , i.e.
 for each $f \in \mathcal{F}$, it holds that $f^{-1}\in \mathcal{F}$. 
Then, one of the 
following possibilities holds \cite{Gh,Nav}:
\begin{enumerate}
  \item [i)] there is a point $x \in S^1$ with finite orbit i.e. $Card(\mathcal{O}^+_{\mathcal{F}}(x)) < \infty$;
  \item [ii)]  $\mathcal{F}$ is a minimal system, i.e. $S^1=\overline{\mathcal{O}^+_{\mathcal{F}}(x)}$, for all $x\in S^1$; or
  \item [iii)] there is a unique  exceptional minimal set $K$  for $\mathcal{F}$, i.e.
  $K$ is a Cantor set $K$ which the following holds:
$f(K)=K$ for all $f\in \mathcal{F}$ and $K =\overline{\mathcal{O}^+_{\mathcal{F}}(x)}$ for all $x \in K$.
\end{enumerate}
Let  $\mathcal{F}=\{f_1:S^1\to S^1|\ i=1,\dots, k\}$ be a symmetric  $\IFS$ consist of finite homeomorphisms 
which admits an exceptional minimal set $K$ such that the orbit of every point of $S^1\setminus K$ is dense in $S^1$. 
Existence  of such $\IFS$ with an exceptional minimal set has been guaranty by Exer. 2.1.5. in  \cite{Nav}.
So, the closed invariant subsets of $S^1$ for $\mathcal{F}=\mathcal{F}_-$ are $\emptyset,~ K$  and $S^1$.

Now, consider any homeomorphism $h$ of $S^1$ such that $h(K)$ 
strictly contains $K$ and $h$ has an attracting fixed point $p\in S^1\setminus K$. 
Then the $\IFS$ generated by $\mathcal{F}_-\bigcup \{h^{-1}\}$ is not minimal but it is 
backward minimal i.e. $\mathcal{F}\bigcup \{h\}$ is minimal.

Since the $\IFS$ generated by $\mathcal{F}\bigcup \{h\}$
 is minimal and $h:S^1\to S^1$ has an attracting  fixed point $p$, the circle $S^1$ is an 
 attractor for  $\mathcal{F}\bigcup \{h\}$, see Theorem A of \cite{Sari24}.
 Moreover,  Theorem A of \cite{Sari24} also implies that the Hutchinson operator of 
 $\mathcal{F}\bigcup \{h\}$ is equicontinuous.
 
By Theorem D of \cite{Sari26},  for every open set $U$ of $S^1$, there exists $n \in \mathcal{N}$  so that $F^n_-(U) =S^1$.
On the other hand, Theorem 3.11 of \cite{Sari26}, implies that $\mathcal{F}_-\bigcup \{h^{-1}\}$ is sensitive.
\end{proof}
 \begin{proof}[Proof of Theorem~\ref{T2}]
  Let $f_1=R_\alpha:S^1\to S^1$ be the rotation by an irrational angle $\alpha$ i.e. $R_\alpha(x)=x+\alpha,\mod 1.$
  Define functions $f_2, f_3$ and $f_4$   whose graphics are determined by the segments:
\begin{align*}
f_2 &:\   [(0 , 0) , (\frac{1}{4}, \frac{1}{8})], ~[(\frac{1}{4}, \frac{1}{8}) ,  (\frac{1}{2}, \frac{1}{2})],~ \text{and} ~
[(\frac{1}{2}, \frac{1}{2}) , (1 , 1)],\\
  f_3 &:\  [(0 , 0) , (\frac{5}{8}, \frac{5}{8})], ~[(\frac{5}{8}, \frac{5}{8}) , (\frac{6}{8}, \frac{7}{8})], ~\text{and}~
[(\frac{6}{8}, \frac{7}{8}), (1,1)],\\
 f_4 &:\  [(0 , 0) , (\frac{1}{2}, \frac{1}{2})] , ~[(\frac{1}{2}, \frac{1}{2}) , (\frac{5}{8}, \frac{3}{4})], ~
[(\frac{5}{8}, \frac{3}{4}), (\frac{7}{8}, \frac{7}{8})],~ \text{and} ~
[(\frac{7}{8}, \frac{7}{8}) , (1 , 1)].
\end{align*}
Clear, $f_3\circ f_4\neq f_4\circ f_3$. So, the $\IFS$, $\mathcal{F}:=\{f_1,f_2,f_3,f_4\}$ is not commutative. 

It is also worth noting that  identity map $id_{S^1}:S^1\to S^1\subset f_2\bigcup f_3\bigcup f_4$ 
and $id_{S^1}:S^1\to S^1\subset f_2^{-1}\bigcup f^{-1}_3\bigcup f_4^{-1}$, as a relation.
Moreover, $f_1$ and $f_2$ are minimal homeomorphisms.
Therefore, the circle $S^1$ is an attractor for $\mathcal{F}$ and $\mathcal{F}_-$, and by \cite{Sari26} 
the Hutchinson operators $\mathcal{F}$ and $\mathcal{F}_- $ are equicontinuous.
 \end{proof}

\end{document}